\documentclass[a4paper,10pt]{amsart}
\usepackage[utf8x]{inputenc}
\usepackage{graphicx, xcolor}
\usepackage{amsmath,amssymb,amsthm,comment,mathtools}
\usepackage[colorlinks,allcolors=blue]{hyperref} % optional
\usepackage{cleveref}
\usepackage{pdfpages}
\usepackage{amsmath}
\usepackage{subfigure,cite}
\usepackage{thm-restate}

\usepackage[a4paper,margin=2.6cm]{geometry}

\newtheorem{thm}{Theorem}[section]

\newtheorem{obs}[thm]{Observation}
\newtheorem{prop}[thm]{Proposition}
\newtheorem{lem}[thm]{Lemma}

\newtheorem{defi}[thm]{Definition}

\providecommand{\RR}{\mathbb{R}}

\title{Sensitivity and Hamming graphs}
 \author[S. Asensio]{Sara Asensio}
 \address{Instituto de Investigaci\'on en Matem\'aticas (IMUVa), Universidad de Valladolid, Valladolid, Spain}
 \email{sara.asensio@uva.es}

 \author[Y. Filmus]{Yuval Filmus}
 \address{The Henry and Marilyn Taub Faculty of Computer Science and Faculty of Mathematics, Technion Israel Institute of Technology, Haifa, Israel.}
 \email{yuvalfi@technion.ac.il}

 \author[I. García-Marco]{Ignacio García-Marco}
 \address{Instituto de Matem\'aticas y Aplicaciones (IMAULL), Secci\'on de Matem\'aticas, Facultad de
Ciencias, Universidad de La Laguna, 38200, La Laguna, Spain}
 \email{iggarcia@ull.edu.es}

\author[K. Knauer]{Kolja Knauer}
\address{Departament de Matemàtiques i Informàtica, Universitat de Barcelona, Barcelona, Spain and Centre de Recerca Matemàtica, Barcelona, Spain}
 \email{kolja.knauer@ub.edu}

\begin{document}

\begin{abstract}
For any $m\geq 3$ we show that the Hamming graph $H(n,m)$ admits an imbalanced partition into $m$ sets, each inducing a subgraph of low maximum degree. This improves previous results by Tandya and by Potechin and Tsang, and disproves the Strong $m$-ary Sensitivity Conjecture of Asensio, García-Marco, and Knauer. On the other hand, we prove their weaker $m$-ary Sensitivity Conjecture by showing that the sensitivity of any $m$-ary function is bounded from below by a polynomial expression in its degree. 
\end{abstract}

\maketitle

\section{Introduction}
In 2019, Huang \cite{H19} provided a one-page proof of the fact that every induced subgraph on more than half of the vertices of the $n$-dimensional hypercube $Q^n$ has maximum degree at least $\sqrt{n}$. Thanks to an equivalence previously obtained by Gotsman and Linial \cite{GL92}, this solved one of the main open problems at that moment in complexity theory: the Sensitivity Conjecture of Nisan and Szegedy \cite{NS94}: $s(f) \geq \sqrt{\deg(f)}$ for any Boolean function $f\colon\{0,1\}^n\to \{0,1\}$ of sensitivity $s(f)$ and degree $\deg(f)$. 

In the same work, Huang proposed to study for a given graph $G$ with nice symmetries, the minimum value of the maximum degree of an induced subgraph on more than $\alpha(G)$ vertices, where $\alpha(G)$ stands for the independence number of $G$. This graph parameter was called the sensitivity of $G$ and denoted by $\sigma(G)$ in \cite{GMK22}. With this notation, Huang's result can be restated as $\sigma(Q^n) \geq \sqrt{n}$. Previously, in 1988, Chung et al.~\cite{CFGS88} constructed subgraphs of $Q^n$ on more than half of the vertices and maximum degree $\lceil \sqrt{n} \rceil$. This construction together with Huang's result prove that the sensitivity of $Q^n$ is $\sigma(Q^n) = \lceil \sqrt{n} \rceil$.

Extending Huang's result to other families of graphs has been an active area of research. His result has been generalized to Cartesian powers of cycles (Tikaradze,~\cite{Tik22}), paths (Zeng and Hou,~\cite{ZH24}), and other Cartesian and semistrong products of graphs (Hong, Lai and Liu,~\cite{HLL20}). Alon and Zheng showed that Huang's result implies a similar result for Cayley graphs over $\mathbb{Z}_2^{\,n}$~\cite{AZ20}, which was later generalized to arbitrary abelian Cayley graphs by Potechin and Tsang~\cite{PT20}, and to Cayley graphs of Coxeter groups and expander graphs by García-Marco and Knauer~\cite{GMK22}. Similar results on Kneser graphs have been developed by Frankl and Kupavskii~\cite{FK20}, and by Chau, Ellis, Friedgut and Lifshitz~\cite{CEFL25}.  On the negative side, infinite families of Cayley graphs with low-degree induced subgraphs on many (more than the independence number) vertices were constructed by Lehner and Verret~\cite{LV20}, and by García-Marco and Knauer~\cite{GMK22}. 

For $m, n \geq 1$,  the Hamming graph $H(n,m)$ is the graph with vertex set $\{0,\ldots,m-1\}^n$ and two vertices are adjacent if and only if they differ in exactly one entry (their Hamming distance is $1$). The Hamming graph $H(1,m)$ is isomorphic to the complete graph $K_m$, and $H(n,2)$ is isomorphic to $Q^n$, the $n$-dimensional hypercube graph. The sensitivity of $H(n,3)$ has been first studied in~\cite{GMK22} and later by Potechin and Tsang~\cite{PT24}. For general $m \geq 3$, Tandya~\cite{T22} exhibits an induced subgraph of $H(n,m)$ with more than $\alpha(H(n,m))$ vertices and maximum degree equal to $1$. Thus, one gets \[ \sigma(H(n,m)) = 
\begin{cases}
    \lceil \sqrt{n} \rceil & \text{if } m = 2, \\
    1 & \text{if } m \ge 3.
\end{cases}
\]

In the case of the hypercube, the results of~\cite{CFGS88} yield a partition $\{V_1,V_2\}$ of the vertices of $Q^n$ such that both sets induce a subgraph of maximum degree at most $\lceil \sqrt{n} \rceil$ and both differ by $1$ from half the vertices. The present paper studies a generalization to Hamming graphs. Namely,  
let $\Pi = \{V_1,\ldots,V_m\}$ be a partition of the vertices of $H(n,m)$ into $m$ sets. The \emph{maximum degree} $\Delta(\Pi)$ of $\Pi$ is the maximum value among the maximum degrees of the induced subgraphs of $H(n,m)$ on $V_1,\ldots,V_m$. The \emph{imbalance} $\iota(\Pi)$ of $\Pi$ is $\sum_{i=1}^m||V_i|-m^{n-1}|$. 
Clearly, if $\Pi$ is \emph{imbalanced}, i.e., $\iota(\Pi)>0$, then $\Delta(\Pi) \geq \sigma(H(n,m))$. Already for $d=1$ our first result strengthens the work of Tandya~\cite{T22} and disproves the Strong $m$-ary Sensitivity Conjecture of~\cite{AGMK24}:

\begin{restatable}[\sf Imbalanced partitions]{thm}{imbalancedpartitions} 
\label{imbalancedpartitions}
For all integers $m,d,n\geq 1$ there exists a partition $\Pi$ of $H(n,m)$ into $m$ sets with maximum degree $\Delta(\Pi)\leq d$ and imbalance \[ \iota(\Pi) = \begin{cases} (m-2) m^{\lfloor \frac{n(d-1)}{d} \rfloor } & \text{if }  d < n \text{ and } m \text{ is even},\\
 (m-1) m^{\lfloor \frac{n(d-1)}{d} \rfloor } & \text{if }  d < n \text{ and } m \text{ is odd}, \\ 2 m^{n-1} \bigg\lfloor \frac{m \lfloor \frac{d}{n} \rfloor }{\lfloor \frac{d}{n} \rfloor + 1 } \bigg\rfloor & \text{if } d\geq n\, .\end{cases} \]
\end{restatable}

Theorem~\ref{imbalancedpartitions} in particular provides large subgraphs with low maximum degree. In Section~\ref{sec:large}, we explore the minimum value of the maximum degree of all induced subgraphs of $H(n,m)$ of a given size. We provide lower bounds for this value using the technique of \emph{supersaturation} and connect this question to results on abelian Cayley graphs~\cite{PT20} and covering codes~\cite{Rod70}.

The study of sensitivity has also been extended beyond Boolean functions without going through graphs. Dafni, Filmus, Lifshitz, Lindzey and Vinyals~\cite{DFLLV21} consider $f\colon \mathcal{X}\rightarrow\{0,1\}$ on different domains such as the symmetric group $\mathcal{X}=S_n$. They show that in this case all classical complexity measures of Boolean functions can also be defined and are polynomially equivalent. In particular, they prove the analogous result to the Sensitivity Conjecture. 
In~\cite{AGMK24} a generalization of the Sensitivity Conjecture to $m$-ary functions, i.e., functions $f\colon \{0,\ldots,m-1\}^n \to \{0,\ldots,m-1\}$, was proposed. This conjecture is implied by the following:

\begin{restatable}[\sf Sensitivity]{thm}{sensitivity}\label{thm.SensitivityConjecture}
Let $n\geq 1$ and $A,B\subseteq \mathbb{R}$ be finite sets and $f\colon A^n \to B$ be a function with sensitivity $s(f)$ and degree $\deg(f)$. Then, $s(f) \geq \sqrt{\frac{\deg(f)}{|A|-1}}$.
\end{restatable}

We close the paper with some open questions in Section~\ref{sec:last}.

\section{Maximum degree of imbalanced partitions of the Hamming graph}

\begin{prop}[{\sf Imbalanced partitions of degree $1$}]
\label{prop:FirstConstruction}
For all integers $m\geq 1, n\geq 2$ there exists a partition $\Pi$ of $H(n,m)$ into $m$ sets with maximum degree $\Delta(\Pi)\leq 1$ and imbalance \[\iota(\Pi)= \begin{cases} m-2 & \text{if } m \text{ is even,}\\
m-1 & \text{if } m \text{ is odd.}\end{cases}\]
\end{prop}
\begin{proof}
    For every $i\in\{0,1,\dots,m-1\}$, let
\[
 S_i = \bigcup_{k=0}^{n-1} \left\{ x  \ \#\ b \ \#\ 0^k : |x| + k + 1 = n,\, b \neq 0,\, \Sigma(x) + \left\lfloor \frac{b+1}{2} \right\rfloor \equiv i \pmod{m} \right\}\ ,
\]
where $\#$ represents concatenation and $\Sigma(x)$ is the sum of the entries of $x$. Moreover, when $x$ is empty  we consider that it sums to $0$. Consider the following induced subgraphs of $H(n,m)$:
    \begin{itemize}
        \item $H_0$ is the induced subgraph on $S_0\cup\{0^n\}$.
        \item $H_i$ is the induced subgraph on $S_i$ for all $i\in\{1,2,\dots,m-1\}.$
    \end{itemize}

For an illustration of the case 
 $m=4$ and $n=3$,
see Figure~\ref{fig:counterexample3}.

    These $m$ induced subgraphs constitute a partition of $H(n,m)$. 
Let us show first that each $S_i$ has maximum degree at most ~$1$. Suppose that $x\ \#\ b\ \#\ 0^k,\ y\ \#\ c\ \#\ 0^\ell \in S_i$ are neighbors. We consider two cases:
\begin{enumerate}
    \item $k = \ell$. In this case $x = y$ and so $\lfloor \frac{b+1}{2} \rfloor = \lfloor \frac{c+1}{2} \rfloor$. Given $b$ there is at most one other option for $c$.
    \item $k \neq \ell$, without loss of generality $k > \ell$. In this case $y= x \ \#\ b \ \#\ 0^{k-\ell-1}$. So
    \[
     \Sigma(x) + \left\lfloor \frac{b+1}{2} \right\rfloor \equiv \Sigma(x) + b + \left\lfloor \frac{c+1}{2} \right\rfloor \pmod{m} \ \Rightarrow\ 
     \left\lfloor \frac{b+1}{2} \right\rfloor - b \equiv \left\lfloor \frac{c+1}{2} \right\rfloor \pmod{m}\ .
    \]
    The left-hand side is $-\left\lceil \frac{b+1}{2} \right\rceil+1$, and so it ranges from $-\left\lceil \frac{m}{2}\right\rceil +1$ to $0$, which is to say from $\left\lfloor \frac{m}{2} \right\rfloor + 1$ to $m$.
    The right-hand side ranges from $1$ to $\left\lfloor\frac{m}{2} \right\rfloor$.
    So the two sides cannot be equal.
\end{enumerate}
The only vertex of $H(n,m)$ not covered by the sets $S_i$ is $0^n$. Its neighbors are of the form $0^i \ \#\ b \ \#\ 0^{n-i-1}$, where $b \neq 0$. In this case $x = 0^i$ and so $\Sigma(x) + \left\lfloor \frac{b+1}{2} \right\rfloor = \left\lfloor \frac{b+1}{2} \right\rfloor \in \{1, \ldots, \left\lfloor \frac{m}{2} \right\rfloor\}$. Therefore $0^n$ does not have neighbors in $S_0$ and hence $H_0$ maintains the maximum degree at most~$1$.

To complete the analysis of the construction, let us compute the sizes of the sets $S_i$:
\[
 |S_i| = (m-1)\sum_{j=1}^{n-1} m^{j-1} + \left|\left\{ b \neq 0 : \left\lfloor \frac{b+1}{2} \right\rfloor = i\right\}\right|.
\]
The first summand equals $m^{n-1} - 1$. When $i = 1$, the second summand is $2$. Hence the partition is imbalanced. In fact, it is possible to compute the exact imbalance of this construction from the value of the second summand in the previous expression.

On the one hand, when $m$ is even, the second summand takes the following values, which lead to an imbalance of $m-2$:

\begin{itemize}
    \item $0$ when $i=0$,
    \item $2$ for $i\in\{1,2,\dots,\frac{m}{2}-1\}$,
    \item $1$ when $i=\frac{m}{2}$, and
    \item $0$ for $i\in\{\frac{m}{2}+1,\dots,m-1\}$.
\end{itemize}

On the other hand, when $m$ is odd, the imbalance is equal to $m-1$. The difference compared to the previous case is that $|\{b\neq 0 : \lfloor\frac{b+1}{2}\rfloor=i\}|$ equals $2$ when $i\in\{1,2,\dots,\frac{m-1}{2}\}$ and $0$ otherwise.\end{proof}

 \begin{figure}[!ht]
 \centering
 	\includegraphics[width=9cm]{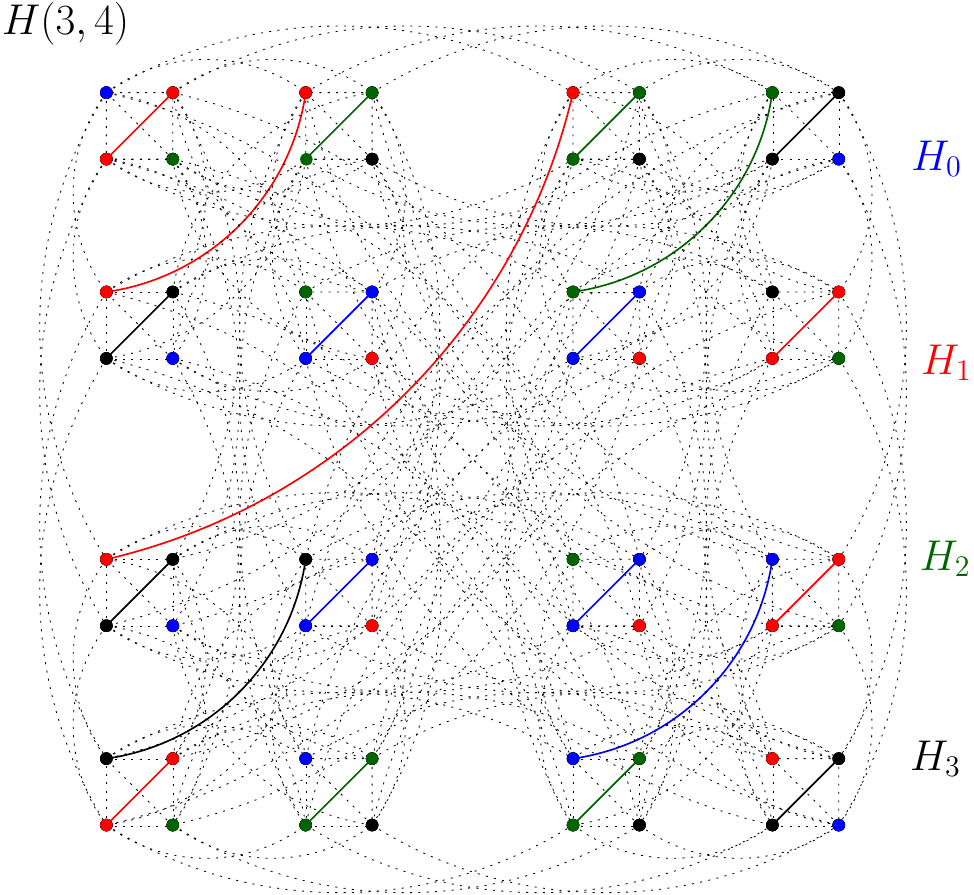}\caption{The construction in Proposition \ref{prop:FirstConstruction} with $m=4$ and $n=3$} 
 	\label{fig:counterexample3}
 \end{figure}

The case $n=1$ for arbitrary $d$ can be analyzed independently.
\begin{lem} 
\label{lm:complete}
For all integers $0 \leq d < m$ there exists a partition $\Pi$ of $K_m$ into $m$ sets with maximum degree $\Delta(\Pi)\leq d$ and imbalance \[\iota(\Pi) = 2 \left\lfloor\frac{d m}{d+1}\right\rfloor.\]

\end{lem}
\begin{proof}
Let $r \in \{0,\ldots,d\}$ be the remainder of the ceiling division between $m$ and $d+1$, that is, $r = \left\lceil \frac{m}{d+1} \right\rceil (d+1) - m$. 

Consider the partition $\Pi$ obtained by arranging the $m$ vertices of $K_m$ between $\left\lceil \frac{m}{d+1} \right\rceil - 1$ sets of size $d+1$, a set of size $(d+1) - r,$ and $m - \left\lceil \frac{m}{d+1} \right\rceil$ empty sets. The maximum degree of such a partition is $d$ and  $\iota(\Pi) = \left(\left\lceil \frac{m}{d+1} \right\rceil - 1\right) d + (d+1) - r - 1 + m - \left\lceil \frac{m}{d+1} \right\rceil = 2\left(m-\left\lceil\frac{m}{d+1}\right\rceil\right) = 2 \left\lfloor\frac{d m}{d+1}\right\rfloor$.
\end{proof}

An idea already present in ~\cite{PT24} allows to lift partitions of given degree and imbalance. 

\begin{lem}\label{lm:lifting}
Let $m\geq 2,\ n\geq d\geq  1$. If for $d'\leq d$ and $n'=\left\lceil\frac{n}{\left\lfloor\frac{d}{d'}\right\rfloor}\right\rceil$ the graph  $H(n',m)$ admits a partition $\Pi'$ into $m$ sets with $\Delta(\Pi')\leq d'$, %and $\iota(\Pi')\geq i'$,
then $H(n,m)$ admits a partition $\Pi$ into $m$ sets with $\Delta(\Pi)\leq d$ and $\iota(\Pi) =  m^{n-n'} \iota(\Pi')$.
\end{lem}
\begin{proof}
Partition $[n]$ into $n'$ sets $P_1, \ldots, P_{n'}$ such that $|P_i|\leq \lceil\frac{n}{n'}\rceil$ for all $i\in [n']$. Define the mapping $\sigma\colon H(n,m)\to H(n',m)$ that maps vertex $x$ to $(\sum_{i\in P_1}x_i, \ldots, \sum_{i\in P_{n'}}x_i)$. We note that:\begin{enumerate}
    \item for all $x'\in H(n',m)$, $|\sigma^{-1}(x')|=m^{n-n'}$,
    \item if $x,y$ are adjacent in $H(n,m)$, then $\sigma(x), \sigma(y)$  are adjacent in $H(n',m)$, and
    \item if $\sigma(x),y'$ are adjacent in $H(n',m)$, then there are at most $\lceil\frac{n}{n'}\rceil$ vertices $y\in \sigma^{-1}(y')$ that are adjacent to $x$ in $H(n,m)$.
\end{enumerate}

Given $\Pi'=\{V'_1, \ldots, V'_m\}$, define $\Pi=\{V_1, \ldots, V_m\}$, where $V_i=\sigma^{-1}(V'_i)$ for all $i\in[m]$. If $x,y$ are adjacent in some $V_i$, then by (2) $\sigma(x), \sigma(y)$  are adjacent in $V'_i$. Hence all neighbors $y$ of $x$ in $V_i$ must arise from a pair $\sigma(x),y'$ of adjacent vertices in $H(n',m)$, where $y\in \sigma^{-1}(y')$. By (3) and since $\sigma(x)$ has at most $d'$ neighbors in $V'_i$, we have that $x$ has at most $d'\lceil\frac{n}{n'}\rceil\leq d$ neighbors in $V_i$. 

Moreover, we can compute \[ \iota(\Pi)=\sum_{i=1}^m||V_i|-m^{n-1}|=\sum_{i=1}^m||\sigma^{-1}(V_i')|-m^{n-1}|=m^{n-n'}\sum_{i=1}^m||V'_i|-m^{n'-1}|=m^{n-n'}\iota(\Pi'). \qedhere \]
\end{proof}

\imbalancedpartitions*
\begin{proof}
Note that for $m=1,2$ there is nothing to show.
If $d<n$ consider the partition $\Pi'$ of $H(\lceil \frac{n}{d}  \rceil,m)$ described in Proposition \ref{prop:FirstConstruction}. By Lemma \ref{lm:lifting}, $H(n,m)$ admits a partition $\Pi$ with $\Delta(\Pi) \leq d$ and $\iota(\Pi) = m^{\lfloor \frac{(d-1) n}{d} \rfloor } \iota (\Pi')$.

If $d  \geq n$, consider the partition $\Pi'$ of $H(1,m) = K_m$ of Lemma \ref{lm:complete} with maximum degree $\lfloor \frac{d}{n} \rfloor$. By Lemma \ref{lm:lifting}, $H(n,m)$ admits a partition $\Pi$ with $\Delta(\Pi) \leq \lfloor \frac{d}{n} \rfloor n \leq d$ and $\iota(\Pi) = m^{n-1} \iota (\Pi') = 2 m^{n-1} \bigg\lfloor \frac{m \lfloor \frac{d}{n} \rfloor }{\lfloor \frac{d}{n} \rfloor + 1 } \bigg\rfloor$.
\end{proof}

\section{The maximum degree of large subgraphs of the Hamming graph}\label{sec:large}

The results of the previous section allow to construct relatively large subgraphs of the Hamming graph with small maximum degree.
Let $m\geq 3$. Given an induced subgraph $G'$ of $H(n',m)$ with $\Delta(G') \leq d'$, the proof of Lemma \ref{lm:lifting} shows how to lift $G'$ to a subgraph $G$ of $H(n,m)$ with $\Delta(G) \leq \lceil\frac{n}{n'}\rceil d'$ and $|V(G)| = |V(G')|\, m^{n-n'}$. Exploiting this idea one gets the following.

\begin{lem}\label{lem:ubprevio}
    Let $m\geq 3$ and $n,d\in\mathbb Z_{> 0}$. There exists an induced subgraph $G$ of $H(n,m)$ with $\Delta(G)\leq d$ and \[|V(G)| = \begin{cases}  m^{n-1} + m^{\lfloor \frac{(d-1)n}{d} \rfloor} & \text{if } d<n,\\
\lceil \frac{d+1}{n} \rceil  m^{n-1}& \text{if } n\leq d\leq (m-1)n\, . \end{cases}  \]
\end{lem}
\begin{proof} If $d < n$, we consider the Hamming graph  $H(n',m)$ with $n' = \lceil \frac{n}{d} \rceil$  and take the induced subgraph $H_1$ of size $m^{n'-1}+1$ and $\Delta(H_1)=1$ described in the proof of Proposition \ref{prop:FirstConstruction}. We can lift $H_1$ to get a subgraph $G$ of $H(n,m)$  of size $m^{n-1} + m^{\lfloor \frac{(d-1)n}{d} \rfloor}$ and $\Delta(G) \leq d$. If $n \leq d \leq (m-1)n$, then consider an induced subgraph of $K_m = H(1,m)$ of size $\lfloor \frac{d}{n}\rfloor+1$ (which has maximum degree equal to $\lfloor \frac{d}{n}\rfloor$) and lift it to $H(n,m)$. This provides a subgraph $G$ of size $\lceil \frac{d+1}{n} \rceil  m^{n-1}$ and $\Delta(G) \leq d$.
\end{proof}

 As a consequence of this lemma, we get the following.

\begin{lem}\label{lem:ub}
    Let $m\geq 3$, $n\geq 1$ and $\epsilon>0$ such that $\epsilon\cdot m^n\in\mathbb N$. There exists an induced subgraph $G$ of $H(n,m)$ on at least $\left(\frac{1}{m}+\epsilon\right)m^n$ vertices and 
\[\Delta(G) \leq \begin{cases}  \lceil\frac{n-1}{ \log_m (1/\epsilon)-1}\rceil & \text{if } \epsilon \leq \frac{1}{m^2},\\
\lceil \epsilon m \rceil n& \text{if } \frac{1}{m^2}  < \epsilon \leq \frac{m-1}{m}\, . \end{cases}  \]
\end{lem}
\begin{proof}
     On the one hand, if $\epsilon\leq \frac{1}{m^2}$ then we take $d=\lceil\frac{n-1}{\log_m(\frac{1}{\epsilon})-1}\rceil < n$. By Lemma \ref{lem:ubprevio}, it suffices to show that $(\frac{1}{m}+\epsilon)m^n\leq m^{n-1}+m^{\lfloor\frac{(d-1)n}{d}\rfloor}$.
     Indeed, since $d\geq \frac{n-1}{\log_m(\frac{1}{\epsilon})-1}$, then $\frac{n-1}{d}+1\leq\log_m(\frac{1}{\epsilon})$ and this implies that $\lceil\frac{n}{d}\rceil\leq\log_m(\frac{1}{\epsilon})$; the latter expression is equivalent to $(\frac{1}{m}+\epsilon)m^n\leq m^{n-1}+m^{\lfloor \frac{(d-1)n}{d}\rfloor}$.

     On the other hand, if $\frac{1}{m^2}\leq\epsilon\leq\frac{m-1}{m}$ then we take $d=\lceil\epsilon m\rceil n$. Again, by Lemma \ref{lem:ubprevio} it suffices to observe that $(\frac{1}{m}+\epsilon)m^n\leq \lceil\frac{d+1}{n}\rceil m^{n-1}$ to conclude the proof.
\end{proof}

When $m = 3$ and $n \geq 6$ one can slightly improve the previous results by lifting the induced subgraph on $3^{n-1}+18$ vertices of $H(n,3)$ and maximum degree $1$ described in \cite{PT24}.

In the rest of the section we provide lower bounds for the maximum degree of an induced subgraph of $H(n,m)$ of given size.

\begin{prop}\label{prop:lb}
    Let $m\geq 2$, $n\geq 1$ and $\epsilon>0$ such that $\epsilon\cdot m^n\in\mathbb N$. Every induced subgraph $G$ of $H(n,m)$ on at least $(\frac{1}{m}+\epsilon)m^n$ vertices has $\Delta(G)\geq \frac{2\epsilon n}{(m-1)(\frac{1}{m}+\epsilon)}$.
\end{prop}

\begin{proof} It suffices to prove the result for $|V(G)|=(\frac{1}{m}+\epsilon)m^n$.

\vspace{1mm}
Let us consider the following way to choose a random edge in $H(n,m)$:
\begin{enumerate}
    \item Choose a random coordinate $i$.
    \item Choose a random assignment to the coordinates other than $i$ to get a copy of $K_m$.
    \item Choose a random edge in the copy of $K_m$.
\end{enumerate}

Let $X$ be the random variable representing the number of vertices in the intersection of the induced subgraph $G$ with a random copy of $K_m$. Then $\mathbb{E}[m-X] = m -\frac{|V(G)|\cdot n}{n\cdot m^{n-1}} = m - 1 - \epsilon m$ and $m - X \geq 0$, and so according to Markov's inequality,
\[
 \Pr[X \leq 1] = \Pr[m - X \geq m - 1] \leq \frac{m - 1 - m\epsilon}{m - 1} = 1 - \frac{m}{m-1} \epsilon \ \Rightarrow\ \Pr[X \geq 2] \geq \frac{m}{m-1} \epsilon.
\]
If $X \geq 2$ then the probability that a random edge of $K_m$ connects two points in $G$ is at least $\frac{1}{\binom{m}{2}}$, and so the probability that a random edge connects two points in $G$ is at least
\[
 \frac{2}{(m-1)^2} \epsilon.
\]
It follows that the average degree of a vertex in $G$ is at least
\[
 \frac{2\cdot \frac{2}{(m-1)^2} \epsilon\cdot |E(H(n,m))|}{|V(G)|}=\frac{\frac{4}{(m-1)^2} \epsilon \cdot \frac{m-1}{2} n m^n}{(\frac{1}{m} + \epsilon) m^n} =
 \frac{2\epsilon n}{(m-1)(\frac{1}{m}+\epsilon)} \ .
\]
In particular, there is a vertex of at least this degree.
\end{proof}

Since $H(n,m)$ is the Cayley graph of an abelian group, by~\cite{PT20} we get that 

\begin{obs}
Let $m\geq 3$ and $n\geq 1$. Every induced subgraph $G$ of $H(n,m)$ on more than $\frac{1}{2}m^n$ vertices has $\Delta(G)\geq \sqrt{\frac{(m-1)n}{2}}$.
\end{obs}

For a graph $G$, denote by $\gamma(G)$ its \emph{domination number}. Then any set on more than $|V(G)|-\gamma(G)$ vertices of a regular graph $G$ induces a subgraph of maximum degree $\Delta(G)$. The domination number of Hamming graphs has been studied, also from the equivalent perspective of covering codes, see the book~\cite{CHLL97}. Using a classical result~\cite{Rod70} one gets:
\begin{obs}
    Let $m\geq 2$ and $n\geq 1$. Every induced subgraph $G$ of $H(n,m)$ on more than $m^n-\max\left(\frac{m^{n-1}}{n-1},\frac{m^n}{(m-1)n+1}\right)$ vertices has $\Delta(G)\geq (m-1)n$.
\end{obs}

\section{Lower bounding sensitivity by degree}

Let $A = \{ a_1, \dots, a_m \} \subset \RR$ and let $B = \{ b_1, \dots, b_k \} \subset \RR$. We consider functions $f\colon A^n \to B$ with inputs $x_1, \dots, x_n$, and we start by recalling some basic notions regarding complexity measures of these functions.

One says that a polynomial $p \in \mathbb R[x_1,\ldots,x_n]$ represents a function $f\colon A^n \to B$ if $p(x) = f(x)$ for every $x \in A^n$. Every function $f\colon A^n \to B$ is represented by a unique polynomial with degree at most $m - 1$ in each variable. Furthermore, this representation minimizes the degree of a polynomial representing $f$ (see, e.g., \cite[Proposition 2.2]{AGMK24}).

\begin{defi}[Degree]
The \emph{degree} of a function $f\colon A^n \to B$ is the degree of the unique polynomial representing $f$ with individual degree at most $m - 1$.
\end{defi}

The \emph{Hamming distance} of two points $x, y\in A^n$ is $|\{i\in[n]\, \colon\,  x_i\neq y_i\}|$. Note  that if we connect points in $A^n$ of Hamming distance $1$  with an edge we obtain $H(n,m)$ and the Hamming distance corresponds to the graph distance. The graph perspective may be comfortable in the following but we will not make it explicit throughout. 

\begin{defi}[Sensitivity]
The \emph{local sensitivity} $s_x(f)$ of a function $f\colon A^n\to B$ at a point $x\in A^n$ is the number of points $y\in A^n$ at Hamming distance 1 from $x$ such that $f(y)\neq f(x)$. 
The \emph{sensitivity} $s(f)$ of a function $f\colon A^n \to B$ is the maximum among the local sensitivities of $f$ at all points of $A^n$.
\end{defi}

The following theorem for $A=B=\{0,1\}$ is Huang's Sensitivity Theorem~\cite{H19}. Here we will show that the general case follows from Huang's result. This in particular confirms the $m$-ary Sensitivity Conjecture from~\cite{AGMK24}. 

\sensitivity*

\begin{proof}
Let $f\colon A^n \to B$ be a function of degree $d$, and let $m=|A|$. We identify $f$ with the polynomial witnessing its degree.

The first step is to reduce the range to $\{0, 1\}$. To this end, let $f_b\colon A^n \to \{0, 1\}$ be defined as $[f(x)=b]$, where here and in the next proof  
\[[i=j]=\begin{cases}
    1 & \text {if } i= j,\\
    0 & \text{otherwise}
\end{cases}\] denotes the \emph{Kronecker delta}. 
Since  $f = \sum_{b \in B} bf_b$, we see that some function $f_b$ has degree at least $d$. Moreover, a sensitive point of $f_b$ has at most the same sensitivity with respect to $f$, since $f_b(x) \neq f_b(y)$ implies $f(x) \neq f(y)$. Hence $s(f_b) \leq s(f)$ and it suffices to lower bound the sensitivity of $f_b$.

Let $D = \left\lceil\frac{d}{m-1}\right\rceil$. Since the degree of $f_b$ is at least $d$ and the individual degree is at most $m-1$, $f_b$~has a monomial involving at least $D$ coordinates. Suppose that one of them is $x_n$. Hence, we can write
\[
 f_b = \sum_{\mu=(\mu_1,\dots,\mu_{n-1})\in\mathbb{N}^{n-1}} x_1^{\mu_1}\cdots x_{n-1}^{\mu_{n-1}}\, P_\mu(x_n),
\]
where $P_\mu(x_n)$ is a function involving only $x_n$, and there is $\mu_0=(\mu_1,\dots,\mu_{n-1})\in\mathbb{N}^{n-1}$ such that $M=x_1^{\mu_1}\cdots x_{n-1}^{\mu_{n-1}}$ involves at least $D - 1$ variables and $P_{\mu_0}$ is not constant, say $P_{\mu_0}(a_s) \neq P_{\mu_0}(a_t)$. 

If we denote by $f_{b,a_i}\colon A^{n-1}\times\{a_i\}\rightarrow\{0,1\}$ the restriction of $f_b$ to the set of points whose last entry is equal to $a_i$, then we can write \[f_b=\sum_{i=1}^m f_{b,a_i}\cdot\left(\prod_{\substack{j=1\\ j\neq i}}^m \frac{x_n-a_j}{a_i-a_j}\right)\, ,\] where 
\begin{align*}
    f_{b,a_i}(x_1,\dots,x_{n-1})&= f_b(x_1,\dots,x_{n-1},a_i)=\sum_{\mu=(\mu_1,\dots,\mu_{n-1})\in\mathbb{N}^{n-1}}x_1^{\mu_1}\cdots x_{n-1}^{\mu_{n-1}}\, P_{\mu}(a_i) \\
    &= c_i\cdot M + \text{other terms in }x_1,\dots,x_{n-1}
\end{align*} for some constant $c_i$. As a consequence \[P_{\mu_0}(x_n)=\sum_{i=1}^mc_i\cdot\left(\prod_{\substack{j=1\\j\neq i}}^m\frac{x_n-a_j}{a_i-a_j}\right)\, ,\] with $c_s\neq c_t$ since $P_{\mu_0}(a_s)\neq P_{\mu_0}(a_t)$.

Restricting the last coordinate of the domain of $f_b$ to $\{a_s, a_t\}$ we get $g_b\colon A^{n-1}\times\{a_s,a_t\}\rightarrow \{0,1\}$, with \[g_b=f_{b,a_s}\cdot\left(\frac{x_n-a_t}{a_s-a_t}\right)+f_{b,a_t}\cdot\left(\frac{x_n-a_s}{a_t-a_s}\right)\, .\]

The coefficient of the monomial $M\cdot x_n$ in the previous expression is \[\frac{c_s}{a_s-a_t}+\frac{c_t}{a_t-a_s}=\frac{c_s-c_t}{a_s-a_t}\neq 0\, ,\] and hence $g_b$ still has a monomial which involves at least $D$ coordinates, which is exactly $M\cdot x_n$. 

Iterating this process, we can find a restriction of $f_b$, obtained by reducing the size of the domain of $D$ coordinates to $2$, whose degree is at least $D$.  Then, we arbitrarily fix the value of the remaining variables. The result is a restriction of the function to a Boolean function on $\{a_{i_1},b_{i_1}\} \times \cdots \times \{a_{i_D},b_{i_D}\}$ of degree $D$ whose sensitivity lower-bounds the sensitivity of the original function.
We can apply Huang's theorem~\cite{H19} to conclude that the restricted function has sensitivity at least $\sqrt{D}$. Hence the same holds for $f_b$ and for $f$, and we conclude that
\[
 s(f) \geq \sqrt{\frac{\deg(f)}{m-1}}\, . \qedhere
\]\end{proof}

In fact, we can show that the bound provided in Theorem \ref{thm.SensitivityConjecture} is almost tight.

\begin{prop}\label{prop:ub}
    There is a function $F\colon A^n\rightarrow B$ such that $s(F)\leq\sqrt{(m-1)\,\deg(F)}$.
\end{prop}
\begin{proof}
    For the sake of the construction assume that $n\geq s^2$ for some integer $s$.
    In the Boolean case, there is a function $f$ with sensitivity $s$ and degree $s^2$ for every $s$, namely a tribes function with $s$ tribes of size $s$. More precisely, for a partition $[s^2]=P_1\cup\dots\cup P_s$ with $|P_i|=s$ for all $i\in[s]$, define $f(x)=\max_{i\in[s]}\min\{x_j\mid j \in P_i\}$. The degree is attained by $1-(\Pi_{i\in[s]}(1-\Pi_{j\in P_i}x_j))$ and the sensitivity by $x = (x_1,\ldots,x_n)$ with $x_j=1$ if $j \in P_1$ and $x_j=0$ otherwise. 

    Let us now fix an element $a\in A$, and define $F\colon A^n\rightarrow B$ as follows: for every $(x_1,\dots,x_n)\in A^n$, $F(x_1,\dots,x_n)$ is equal to the value that $f$ takes at the vector obtained from $(x_1,\dots,x_n)$ by replacing each input $x_i$ with the function $[x_i=a]$; i.e.,
    \[ F(x_1,\dots,x_n)=f([x_1=a],\dots,[x_n=a])\,.\]

    By construction $s(F) \leq (m-1) s(f)$. % Also, this function satisfies $s(F)\geq s_{(a,\dots,a)}(F)  = (m-1)\, s_{\mathbf{1}}(f)=(m-1)\, s$, and
    Also, $\deg(F) = (m - 1)\, s^2$ since the function $[x_i=a]$ is represented by the polynomial $\prod_{\substack{j=1\\ j\neq i}}^m\frac{x_i-a_j}{a-a_j}$ and has degree $m-1$. Hence
\[
 s(F) \leq \sqrt{(m-1) \deg(F)}. \qedhere
\]\end{proof}

\section{Questions and conjectures}\label{sec:last}

Our work leaves two gaps that we would like to see closed:
\begin{itemize}
    \item Concerning large subgraphs of small maximum degree there remains an exponential gap. In Proposition~\ref{prop:lb} we show that an induced graph of $H(n,m)$ with a $(\frac{1}{m}+\epsilon)$-fraction of the vertices has maximum degree at least $\Omega_m(\epsilon) n$. On the other hand in Lemma~\ref{lem:ub} we construct such graphs with maximum degree at most $\frac{1}{\log_m(\frac{1}{\epsilon})} \cdot n$.
    \item Concerning sensitivity of $m$-ary functions there remains a small gap between the lower bound of $\sqrt{\frac{\deg(f)}{m-1}}$ from Theorem~\ref{thm.SensitivityConjecture} and the upper bound of $\sqrt{(m-1)\deg(f)}$ from Proposition~\ref{prop:ub}.
\end{itemize}

\section*{Acknowledgements}

Yuval Filmus was supported by ISF grant 507/24.
Sara Asensio, Ignacio García-Marco, and Kolja Knauer were supported by 
 PID2022-137283NB-C22 funded by MICIU/AEI/10.13039/501100011033 and by ERDF/EU. Sara Asensio was also supported by European Social Fund, \textit{Programa Operativo de Castilla y León}, and \textit{Consejería de Educación de la Junta de Castilla y León}. Kolja Knauer was also supported by the Spanish State Research Agency, through the Severo Ochoa and María de Maeztu Program for Centers and Units of Excellence in R\&D (CEX2020-001084-M). Part of this work was done during a research visit of Sara Asensio to Universidad de La Laguna funded by the MICIU grant RED2022-134947-T.

The authors thank the anonymous referees for their feedback and suggestions that have helped to correct several inaccuracies and have improved the manuscript's quality.

\bibliographystyle{plain}
\bibliography{biblio}
    
\end{document}